\documentclass[draft,reqno]{amsproc}
\usepackage{latexsym,amssymb,amsmath,color,pifont}

\newtheorem{thm}{Theorem}
\newtheorem{lem}[thm]{Lemma}

\newtheorem*{opqq}{Question}
\newtheorem{pro}[thm]{Proposition}
\newtheorem{cor}[thm]{Corollary}

\theoremstyle{remark}
\newtheorem{rem}[thm]{Remark}
\newtheorem{exa}[thm]{Example}

\DeclareMathOperator{\supp}{supp}

\newcommand*{\abold}{\boldsymbol{a}}
\newcommand*{\borel}[1]{\mathfrak{B}(#1)}
\newcommand*{\okr}{\stackrel{\scriptscriptstyle{\mathsf{def}}}{=}}
\newcommand*{\gammab}{\boldsymbol \gamma}
\newcommand*{\Gammab}{\boldsymbol{\varGamma}}

\newcommand*{\cbb}{\mathbb C}
\newcommand*{\D}{\mathrm d}
\newcommand*{\E}{\mathrm e}
\newcommand*{\Ge}{\geqslant}
\newcommand*{\I}{\mathrm i}
\newcommand*{\IM}{\mathrm{Im}}
\newcommand*{\Le}{\leqslant}

\newcommand*{\nfr}{\mathfrak N}
\newcommand*{\nul}{\{0\}}
\newcommand*{\PDE}{\mathsf{PDE}}
\newcommand*{\rbb}{\mathbb R}
\newcommand*{\RE}{\mathrm{Re}}
\newcommand*{\sbold}{\boldsymbol{s}}
\newcommand*{\tbb}{\mathbb T}
\newcommand{\tbold}{\boldsymbol{t}}
\newcommand*{\zbb}{\mathbb Z}

\newcommand*{\zcal}{\mathcal Z}

\begin{document}
   \title[The complex moment problem: determinacy and
extendibility]{The complex moment problem: determinacy
and extendibility}
   \author[D. Cicho\'n, J. Stochel, F.H. Szafraniec]
{Dariusz Cicho\'n, Jan Stochel, Franciszek Hugon
Szafraniec}
   \address{Wydzia\l\ Matematyki i Informatyki, Uniwersytet
Jagiello\'nski, ul.\ {\L}ojasiewicza 6, PL-30348
Kra\-k\'ow}
   \email{Dariusz.Cichon@im.uj.edu.pl}
   \email{Jan.Stochel@im.uj.edu.pl}
   \email{umszafra@cyfronet.pl}
   \thanks{This work was supported by the NCN
(National Science Center), decision No.\
DEC-2013/11/B/ST1/03613}
   \begin{abstract}
Complex moment sequences are exactly those which admit
positive definite extensions on the integer lattice
points of the upper diagonal half-plane. Here we prove
that the aforesaid extension is unique provided the
complex moment sequence is determinate and its only
representing measure has no atom at $0$. The question
of converting the relation is posed as an open
problem. A partial solution to this problem is
established when at least one of representing measures
is supported in a plane algebraic curve whose
intersection with every straight line passing through
$0$ is at most one point set. Further study concerns
representing measures whose supports are Zariski dense
in $\cbb$ as well as complex moment sequences which
are constant on a family of parallel ``Diophantine
lines''. All this is supported by a bunch of
illustrative examples.
   \end{abstract}
   \maketitle

There are two ways of approaching the complex moment
problem (see \cite{bcr}; for a recent survey of the
complex moment problem see also \cite{Schmu}). One
following an idea due to Marcel Riesz (for
continuation see \cite{hav1,hav2,Kilpi}) and the other
via positive definite extendibility (see
\cite{st-sz1,c-s-s-tr}). As is well-known, positive
definiteness is not sufficient for solving the complex
moment problem (see \cite{Schmu0,bcr}). The present
paper carries on with the study of \cite{st-sz1} which
characterizes solving the complex moment problem by
extending a given sequence defined on the integer
lattice points of the first quarter to a positive
definite sequence on the lattice points of the upper
diagonal half-plane. One may expect a relationship
between the uniqueness of extending sequence on one
hand and the determinacy of the resulting moment
sequence. This question leads to quite a number of
interesting thoughts which are exposed in this paper.
Our results, which are diverse in nature, are
supported by elucidative examples and lead eventually
to an open problem discussed on the final pages of the
paper.
   \section{Introduction}
In this paper $\borel{Z}$ stands for the
$\sigma$-algebra of all Borel subsets of a topological
Hausdorff space $Z$. With the notation
   \begin{align*} \nfr
\okr \{(m,n)\colon m,n \text{ are integers such that }
m\geqslant 0, n\geqslant 0\},
   \\
\nfr_+ \okr \{(m,n)\colon m,n \text{ are integers such
that } m+n\geqslant 0\},
   \end{align*}
we say that a sequence $\gammab =
\{\gamma_{m,n}\}_{(m,n) \in \nfr}\subset \cbb$ is a
{\em complex moment sequence} if there exists a Borel
measure\footnote{\;All measures considered in this
paper are positive. We always tacitly assume that
integrands are absolutely integrable wherever they
appear.} $\mu$ on $\cbb$ such that
   \begin{align} \label{1}
\gamma_{m,n} = \int_\cbb z^m\bar z^n \D \mu(z), \quad
(m,n) \in \nfr;
   \end{align}
recall that $\rbb$ and $\cbb$ stand for the fields of
all real and complex numbers, respectively. We call
the measure $\mu$ a {\em representing measure} for the
sequence $\gammab$. If $\mu$ in \eqref{1} is unique,
then the sequence $\gammab$ is said to be {\em
determinate}\,\footnote{\;This is one of the three
determinacy notions considered in \cite{Fug}.}. As is
easily seen, a necessary condition for $\gammab$ to be
a complex moment sequence is that $\gammab$ is {\em
positive definite on $\nfr$}, that is
   \begin{align*}
\sum_{(m,n), (p,q) \in \nfr}
\lambda_{m,n}\bar\lambda_{p,q} \gamma_{m+q,n+p}
\geqslant 0
   \end{align*}
for every sequence $\{\lambda_{m,n}\}_{(m,n) \in \nfr}
\subset \cbb$ vanishing off a finite set. The above
positive definiteness condition is in general not
sufficient. However, it turns out that complex moment
sequences are exactly those which admit positive
definite extensions on $\nfr_+$ (see \cite[Theorem
1]{st-sz1}). More precisely, $\gammab =
\{\gamma_{m,n}\}_{(m,n) \in \nfr}\subset \cbb$ is a
complex moment sequence if and only if there exists a
sequence $\Gammab = \{\varGamma_{m,n}\}_{(m,n)\in
\nfr_+} \subset \cbb$ which is {\em positive definite
on $\nfr_+$}, that is
   \begin{align*}
\sum_{(m,n), (p,q) \in \nfr_+}
\lambda_{m,n}\bar\lambda_{p,q} \varGamma_{m+q,n+p}
\geqslant 0
   \end{align*}
for every sequence $\{\lambda_{m,n}\}_{(m,n) \in
\nfr_+} \subset \cbb$ vanishing off a finite set, and
which extends $\gammab$, that is
   \begin{align*}
\varGamma_{m,n} = \gamma_{m,n}, \quad (m,n) \in \nfr.
   \end{align*}
Using the notation
   \begin{align*}
\PDE(\gammab) \okr \{\Gammab \colon \Gammab \text{ is
a positive definite extension of } \gammab \text{ on }
\nfr_+\},
   \end{align*}
we can simply rewrite \cite[Theorem 1]{st-sz1} as
follows.
   \begin{thm} \label{nonemp}
A sequence $\gammab = \{\gamma_{m,n}\}_{(m,n) \in
\nfr}\subset \cbb$ is a complex moment sequence if and
only if $\PDE(\gammab) \neq \varnothing$.
   \end{thm}
The main question of this paper concerns a connection
between the following two statements:
   \begin{enumerate}
\item[(i)] $\gammab$ is a determinate complex moment sequence on $\nfr$,
\item[(ii)] $\PDE(\gammab)$ is a singleton.
   \end{enumerate}
According to Proposition \ref{snu2}, if $\gammab$ has
a representing measure $\mu$ such that $\mu(\{0\}) >
0$, then $\PDE(\gammab)$ is infinite. We will show
that (i) implies (ii) provided the representing
measure of $\gammab$ vanishes at $\{0\}$ (see Theorem
\ref{null}). The implication (ii)$\Rightarrow$(i)
holds when $\gammab$ has a representing measure
supported in a real algebraic set belonging to a
distinguished class of plane algebraic curves (see
Theorem \ref{0notatom} and Corollary \ref{nullab}). It
is an open problem whether (ii) implies (i) in full
generality (see Section \ref{Sec6}).

We recall that in view of \cite[Proposition 6]{st-sz1}
a sequence $\Gammab=\{\varGamma_{m,n}\}_{(m,n)
\in\nfr_+} \subset \cbb$ is positive definite if and
only if there exist two Borel measures $\mu$ on
$\cbb^*\okr \cbb\setminus\nul$ and $\nu$ on $\mathbb
T$ (the unit circle centered at the origin) such that
   \begin{align} \label{repr1}
\varGamma_{m,n} = \int_{\cbb^*} z^m\bar z^n \D\mu(z) +
\delta_{m+n,0} \int_{\mathbb T} z^m\bar z^n \D\nu(z),
\quad (m,n) \in \nfr_+,
   \end{align}
   where $\delta_{k,l} = 1$ if $k=l$ and $\delta_{k,l}
= 0$ otherwise. If \eqref{repr1} holds, we say that
$(\mu,\nu)$ is a {\em representing} pair for
$\Gammab$. If such a pair is unique, then $\Gammab$ is
called {\em determinate}. Depending on circumstances,
we will identify a Borel measure $\mu$ on $\cbb^*$
with a Borel one on $\cbb$ vanishing on $\{0\}$.

In this paper the notation $\gammab =
\{\gamma_{m,n}\}_{(m,n)\in\nfr}$ will be used
interchangeably with $\gammab\colon\nfr\to\cbb$; the
same applies to $\Gammab$ and $\nfr_+$.
   \section{Determinacy from extendibility}
In this section we investigate the interplay between
the determinacy of a complex moment sequence $\gammab$
and special properties of the set $\PDE(\gammab)$
including those related to its cardinality.

As shown Lemma \ref{repr2} a representing measure for
a complex moment sequence can be retrieved from a
representing pair for its positive definite extension.
Below, $\delta_a$ stands for the Dirac measure at $a$
understood as the Borel measure on $\cbb$ of total
mass $1$ at the point $a\in\cbb$.
   \begin{lem}\label{repr2}
Suppose $\gammab$ is a complex moment sequence. Then
the following assertions hold{\em :}
   \begin{enumerate}
   \item[(i)] if $(\mu,\nu)$ is a representing pair for some
$\Gammab\in\PDE(\gammab)$, then $\mu + \nu(\mathbb
T)\delta_0$ is a representing measure for $\gammab$,
   \item[(ii)] if $\mu$ is a representing measure for $\gammab$ and
$\mu(\{0\})=0$, then $(\mu,0)$ is a representing pair
for some $\Gammab \in \PDE(\gammab)$.
   \end{enumerate}
   \end{lem}
   \begin{proof}
   (i) Since for every $(m,n)\in \nfr$, the second
term in \eqref{repr1} coincides with the integral of
$z^m\bar z^n$ over $\cbb$ with respect to the measure
$\nu(\mathbb T)\delta_0$, we get (i).

   (ii) Since, by our assumption, the function
$\cbb^*\ni z\to z^m\bar z^n\in \cbb$ is absolutely
integrable for every $(m,n)\in \nfr_+$, we can define
the sequence $\Gammab=\{\varGamma_{m,n}\}_{(m,n)
\in\nfr_+}$ by \eqref{repr1} with $\nu=0$. Then
$\Gammab\in \PDE(\gammab)$ and $(\mu,0)$ is a
representing pair for $\Gammab$.
   \end{proof}
From now on $\varphi$ will denote the continuous
function
   \begin{align} \label{odw}
\varphi\colon \tbb \to \tbb, \quad \varphi(z)=z^2,
\quad z \in \tbb.
   \end{align}
If $\gammab$ is a complex moment sequence, then an
extension $\Gammab \in \PDE(\gammab)$ is called {\em
quasi-determinate} if for any two representing pairs
$(\mu_1,\nu_1)$ and $(\mu_2,\nu_2)$ for $\Gammab$ we
have
   \begin{align*}
\mu_1=\mu_2 \quad \text{and} \quad
\nu_1\circ\varphi^{-1} = \nu_2\circ\varphi^{-1},
   \end{align*}
where $\nu_j\circ\varphi^{-1}$ is the transport of the
measure $\nu_j$ via $\varphi$ given by
   \begin{align} \label{dodo1}
(\nu_j\circ\varphi^{-1})(\sigma) \okr
\nu_j(\varphi^{-1}(\sigma)), \quad \sigma \in
\borel{\tbb}, \, j=1,2.
   \end{align}
This notion appears to be very natural as the
following result shows.

Now we are ready to clarify the role played by
determinacy in the question of uniqueness of positive
definite extensions. This is in a sense the basic
statement.
   \begin{thm}\label{semi}
Let $\gammab\colon\nfr\to\cbb$ be a complex moment
sequence. Then the following conditions are
equivalent:
\begin{enumerate}
\item[(i)] $\gammab$ is  determinate,
\item[(ii)] if $(\mu_1,\nu_1)$ and $(\mu_2,\nu_2)$ are representing pairs for $\Gammab_1,\Gammab_2\in\PDE(\gammab)$, respectively, then
$\mu_1=\mu_2$,
\item[(iii)] if $(\mu_1,\nu_1)$ and $(\mu_2,\nu_2)$ are representing pairs for $\Gammab_1,\Gammab_2\in\PDE(\gammab)$, respectively, then
$\mu_1=\mu_2$ and $\nu_1(\mathbb T)=\nu_2(\mathbb T)$.
   \end{enumerate}
Moreover, if {\em (i)} holds, then every
$\Gammab\in\PDE(\gammab)$ is quasi-determinate.
   \end{thm}
   \begin{proof}
(i)$\Rightarrow$(iii). The determinacy of $\gammab$
and Lemma \ref{repr2} yield
   \begin{align*}
\mu_1 + \nu_1(\mathbb T)\delta_0 = \mu_2 + \nu_2(\mathbb T)\delta_0.
   \end{align*}
By the definition of a representing pair for an
element of $\PDE(\gammab)$, both measures $\mu_1$ and
$\mu_2$ have no atom at $0$. Hence, it follows that
$\mu_1=\mu_2$ and $\nu_1(\mathbb T)=\nu_2(\mathbb T)$.

(iii)$\Rightarrow$(ii). Obvious.

(ii)$\Rightarrow$(i). Suppose that $\lambda_1$ and
$\lambda_2$ are representing measures for $\gammab$.
Note that the measures $\lambda_1$ and $\lambda_2$
have the following unique decompositions
   \begin{align} \label{lam2r}
\lambda_1=\mu_1+t_1\delta_0\quad \text{and} \quad
\lambda_2=\mu_2+t_2\delta_0,
   \end{align}
where $\mu_1$ and $\mu_2$ are Borel measures on
$\cbb^*$ and $t_1,t_2$ are nonnegative real numbers.
Then one can verify that $(\mu_1,t_1\delta_1)$ and
$(\mu_2,t_2\delta_1)$ are representing pairs for some
$\Gammab_1, \Gammab_2 \in \PDE(\gammab)$,
respectively. It follows from (ii) that $\mu_1=\mu_2$,
and thus
   \begin{align*}
t_j \overset{\eqref{lam2r}}= \lambda_j(\cbb) -
\mu_j(\cbb^*) = \gamma_{0,0} - \mu_1(\cbb^*), \quad
\quad j=1,2,
   \end{align*}
which, by \eqref{lam2r} again, implies that
$\lambda_1=\lambda_2$.

To prove the ``moreover'' part assume that (i) holds.
Let $(\mu_1,\nu_1)$ and $(\mu_2,\nu_2)$ be
representing pairs for the same positive definite
extension of $\gammab$ on $\nfr_+$. Then, by (ii),
$\mu_1=\mu_2$, and consequently, by \eqref{repr1} with
$m+n=0$, we have
   \begin{align*}
\int_{\mathbb T} z^{2m} \D\nu_1(z) = \int_{\mathbb T}
z^{2m} \D\nu_2(z),\quad m=0,\pm 1, \pm 2, \dots.
   \end{align*}
Applying the measure transport theorem yields
   \begin{align*}
\int_{\mathbb T} z^m \D\nu_1 \circ \varphi^{-1}(z) =
\int_{\mathbb T} z^m \D\nu_2 \circ \varphi^{-1} (z),
\quad m=0,\pm 1, \pm 2, \dots.
   \end{align*}
Hence, by the determinacy of Herglotz moment problem
(see Section \ref{Sec3}), $\nu_1 \circ
\varphi^{-1}=\nu_2 \circ \varphi^{-1}$, which
completes the proof.
   \end{proof}
   \begin{cor} \label{snu}
Let $\gammab\colon \nfr\to\cbb$ be a determinate
complex moment sequence. Then for every nonzero finite
Borel measure $\nu$ on $\mathbb T$, there exists a
unique triplet $(\mu,s,\Gammab)$ such that $\mu$ is a
Borel measure on $\cbb^*$, $s$ is a nonnegative real
number, $\Gammab\in\PDE(\gammab)$ and $(\mu,s\nu)$ is
a representing pair for $\Gammab$.
\end{cor}
\begin{proof}
Let $\varrho$ be a representing measure for $\gammab$.
Set $s=\varrho(\{0\})/\nu(\tbb)$ and $\mu=\varrho -
\varrho(\{0\}) \delta_0$. It is easily seen that
$(\mu,s\nu)$ is a representing pair for some $\Gammab
\in \PDE(\gammab)$ (cf.\ \eqref{repr1}). The
uniqueness is a direct consequence of Theorem
\ref{semi}.
\end{proof}
One may illustrate Corollary \ref{snu} by considering
particular choices, extreme in a sense, of the measure
$\nu$: $\nu$ being the Lebesgue measure on $\mathbb T$
or with $\nu$ being the Dirac measure $\delta_1$ on
$\tbb$ at $1$.
   \begin{pro}\label{snu2}
Let $\gammab\colon \nfr\to\cbb$ be a complex moment
sequence which has a representing measure $\mu$ such
that
   \begin{align} \label{munul}
\mu(\nul)> 0.
   \end{align}
Then the cardinality of $\PDE(\gammab)$ is equal to
$\boldsymbol{\mathfrak c}$.
   \end{pro}
   \begin{proof}
Set $\alpha=\mu(\nul)$. It is a simple matter to
verify that the sequences $\Gammab_1$ and $\Gammab_2$
on $\nfr_+$ defined via \eqref{repr1} by the pairs
$(\mu - \alpha \delta_0, \alpha\delta_1)$ and $(\mu-
\alpha \delta_0, \alpha \delta_{\I})$, respectively,
are both in $\PDE(\gammab)$; here $\I = \sqrt{-1}$.
Then any convex combination of $\Gammab_1$ and
$\Gammab_2$ is in $\PDE(\gammab)$. Using \eqref{repr1}
with $m+n=0$, where $m$ is odd, we deduce that
$\Gammab_1 \neq \Gammab_2$. As a consequence, the
cardinality of $\PDE(\gammab)$ is at least
$\boldsymbol{\mathfrak c}$. On the other hand, because
$\PDE(\gammab) \subset \cbb^{\nfr_+}$ and the
cardinality of $\cbb^{\nfr_+}$ is
$\boldsymbol{\mathfrak c}$, we conclude that the
cardinality of $\PDE(\gammab)$ is equal to~
$\boldsymbol{\mathfrak c}$.
   \end{proof}
It is clear that in the case of a determinate complex
moment sequence the zero may or may not be an atom of
its representing measure and both instances can occur
(e.g., the sequences $(1,0,0,\ldots)$ and
$(1,1,1,\ldots)$ are complex moment sequences with the
representing measures $\delta_0$ and $\delta_1$,
respectively). One can ask whether the same is true
for indeterminate complex moment sequences. This
question is answered in the affirmative in Section
\ref{Sec3} (see Proposition \ref{Ham-c} and the
subsequent parts).

If $\gammab\colon \nfr\to\cbb$ is determinate complex
moment sequence and $\mu$ is its representing measure
such that $\mu(\nul)=0$, then $\PDE(\gammab)$ is a
single point set, as the following Theorem shows.
   \begin{thm}\label{null}
Let $\gammab$ be a complex sequence defined on $\nfr$.
Then the following statements are equivalent{\em :}
   \begin{enumerate}
   \item[(i)] $\gammab$ is a determinate complex
moment sequence with a representing measure $\mu$ such
that $\mu(\nul)=0$,
   \item[(ii)] $\PDE(\gammab) = \{\Gammab\}$ and $\Gammab$
has the property that $\mu_1=\mu_2$ whenever
$(\mu_1,0)$ and $(\mu_2,0)$ are representing pairs for
$\Gammab$,
   \item[(iii)] $\PDE(\gammab) = \{\Gammab\}$ and $\Gammab$
is determinate.
   \end{enumerate}
Moreover, if $\Gammab$ is as in {\em (iii)}, then
$(\mu,0)$ is its representing pair, where $\mu$ is as
in~{\em (i)}.
   \end{thm}
   \begin{proof}
(i)$\Rightarrow$(iii) Note that $\PDE(\gammab) \neq
\varnothing$ because, by Lemma \ref{repr2}, $(\mu,0)$
is a representing pair for some $\Gammab \in
\PDE(\gammab)$. Let $(\mu_1,\nu_1)$ and
$(\mu_2,\nu_2)$ be representing pairs for some
extensions $\Gammab_1,\Gammab_2\in\PDE(\gammab)$,
respectively. It follows from Theorem~\ref{semi} that
$\mu_1=\mu_2$. In turn, Lemma~\ref{repr2} yields
   \begin{align*}
\mu=\mu_1 + \nu_1(\mathbb T)\delta_0 = \mu_2 +
\nu_2(\mathbb T)\delta_0,
   \end{align*}
which by our assumption $\mu(\nul)=0$ gives
$\nu_1(\mathbb T)=\nu_2(\mathbb T)=0$. Thus the pairs
$(\mu_1,\nu_1)$ and $(\mu_2,\nu_2)$ are equal to
$(\mu,0)$. As a consequence, $\PDE(\gammab) =
\{\Gammab\}$, the extension $\Gammab$ is determinate
and $(\mu,0)$ is a representing pair for $\Gammab$.

(iii)$\Rightarrow$(ii) Obvious.

(ii)$\Rightarrow$(i) It follows from Theorem
\ref{nonemp} that $\gammab$ is a complex moment
sequence. According to Proposition \ref{snu2}, if
$\mu$ is a representing measure for $\gammab$, then
$\mu(\nul) = 0$, and consequently, by Lemma
\ref{repr2}, $(\mu,0)$ is a representing pair for
$\Gammab$. Hence, by the property of $\Gammab$ assumed
in (ii), the complex moment sequence $\gammab$ is
determinate. This completes the proof.
   \end{proof}
   \section{\label{Sec3} Special classes of complex
moment sequences} In the previous section the question
of the cardinality of $\PDE(\gammab)$ was successfully
answered except for the case of an indeterminate
complex moment sequence $\gammab$ which has no
representing measure with atom at $0$. The question
arises as to whether such a $\gammab$ may exist. It is
well-known that every indeterminate Hamburger moment
sequence has a representing measure with an atom at
$0$ (see \cite[Theorem 2.13]{Sh-Tam}). It turns out
that in the case of complex moment sequences this does
not have to be the case (see Example \ref{no-atom}).

We begin by stating a well-known fact on supports of
representing measures of a complex moment sequence
having at least one representing measure supported in
a real algebraic set. In what follows, $\supp \mu$
stands for the closed support of a finite Borel
measure on a metric space (see \cite[Theorem
II.2.1]{Part}). In this paper, by ``support'' we
always mean ``closed support''. Recall that a set
$A\subset \cbb$ is called a {\em real algebraic set}
if $A=\zcal_{p}$ for some $p\in \cbb[z,\bar z]$, where
   \begin{align*}
\zcal_{p} \okr \{z\in\cbb\colon p(z,\bar z)=0\}
   \end{align*}
and $\cbb[z,\bar z]$ stands as usual for the ring of
all polynomials in two indeterminates with complex
coefficients.
   \begin{pro}\label{inv_supp}
If a complex moment sequence $\gammab$ has a
representing measure supported in a real algebraic
subset $A$ of $\cbb$, then all the other representing
measures for $\gammab$ are also supported in $A$.
\end{pro}
\begin{proof}
If $A=\zcal_{p}$, where $p\in \cbb[z,\bar z]$, $\mu_1$
and $\mu_2$ are representing measures for $\gammab$
and $\mu_1$ is supported in $A$, then
   \begin{align*}
\int_\cbb |p(z,\bar z)|^2\D\mu_2(z) = \int_\cbb
|p(z,\bar z)|^2\D\mu_1(z) = \int_A |p(z,\bar
z)|^2\D\mu_1(z)=0,
   \end{align*}
which implies that $\mu_2(\cbb\setminus A)=0$. This
completes the proof.
   \end{proof}
Necessary and sufficient conditions for a complex
moment sequence to have a representing measure
supported in a given plane algebraic curve were given
in \cite{Sto,st-sz0}. Below we discuss the interplay
between a linear Diophantine relation imposed on
indices of a complex moment sequence and supports of
its representing measures. Given integers $k$ and $l$
such that $k\Ge 0$, a sequence
$\gammab=\{\gamma_{m,n}\}_{(m,n)\in\nfr}\subset \cbb$
is called {\em $(k,l)$-flat} if the following
condition holds
   \begin{gather} \notag
   \gamma_{m,n} = \gamma_{m',n'} \text{ whenever }
(m,n),(m',n')\in\nfr \text{ and }
   \\ \label{kaka}
   km+ln=km'+ln'.
   \end{gather}
Intuitively speaking, the $(k,l)$-flatness of
$\gammab$ means that the sequence $\gammab$ is
constant on each ``Diophantine line''
$\{(m,n)\in\zbb^2 \colon km+ln=c\}$ with $c\in\zbb$
(as usual $\zbb$ stands for the set of all integers).
   \begin{thm}\label{unif}
If $\gammab=\{\gamma_{m,n}\}_{(m,n)\in\nfr}$ is a
$(k,l)$-flat complex moment sequence for some $(k,l)$
and $\mu$ is its representing measure, then one of the
following conditions holds{\em :}
   \begin{enumerate}
\item[(i)] $\supp \mu \subset \{0\} \cup G_r$ for
some integer $r\Ge 1$, where $G_r=\{z\in \cbb\colon
z^r=1\}$,
\item[(ii)] $\supp \mu\subset\tbb$,
\item[(iii)] $\supp \mu\subset \rbb$.
   \end{enumerate}
Moreover, if $k\neq |l|$ $($resp., $k=-l$, $k=l$$)$,
then the condition {\em (i)} $($resp., {\em (ii)},
{\em (iii)}$)$ holds. If additionally $k\neq |l|$ and
$l \Le 0$, then $0\notin \supp \mu$.

Reversely, if $\gammab$ is a complex moment sequence
with a representing measure $\mu$ satisfying one of
the conditions {\em (i)-(iii)}, then $\gammab$ is
$(k,l)$-flat, where
   \begin{align} \label{ptaszki}
(k,l) =
   \begin{cases}
   (1,1) & \text{if {\em (i)} holds} \text{ and } r =
1,
   \\
   (1,r-1) & \text{if {\em (i)} holds} \text{ and } r
> 1,
   \\
   (1,-1) & \text{if {\em (ii)} holds,}
   \\
   (1,1) & \text{if {\em (iii)} holds.}
   \end{cases}
   \end{align}
   \end{thm}
   \begin{proof}
Assume that $\gammab$ is a $(k,l)$-flat complex moment
sequence for some $(k,l)$ and $\mu$ is its
representing measure. In the case of $k=l=0$ the
measure $\gamma_{0,0}\delta_1$ is a (necessarily
unique) representing measure for $\gammab$. Hence, its
support satisfies the conditions (i)-(iii).

Suppose $l > 0$. Note that the pairs $(m,n)=(l,l)$ and
$(m',n')=(0,k+l)$ satisfy \eqref{kaka}. The same is
true for the pairs $(k+l,0)$ and $(k,k)$. Hence, by
the $(k,l)$-flatness of $\gammab$, we have
   \begin{align*}
\int_\cbb |z^l-\bar z^k|^2\D\mu(z) = \gamma_{l,l} -
\gamma_{k+l,0} - \gamma_{0,k+l} + \gamma_{k,k} = 0,
   \end{align*}
and so
   \begin{align} \label{hopdzis}
\text{$\supp \mu \subset \zcal_p$ with $p(z,\bar z)=
z^l-\bar z^k$.}
   \end{align}
If $k\neq l$, then the equality $z^l=\bar z^k$ implies
that $|z|=1$ for $z\neq 0$ and, consequently,
$z^{k+l}=1$. This shows that $\zcal_p \setminus \{0\}
\subset G_{k+l}$. Therefore, $\mu$ satisfies (i). In
turn, if $k=l$, then $\gammab$ is $(1,1)$-flat.
Applying \eqref{hopdzis} to $k=l=1$, we see that
$\supp \mu \subset \zcal_p=\rbb$, which means that
$\mu$ obeys (iii).

Now consider the case $l<0$. Note that the pairs
$(m,n)=(k-l,k-l)$ and $(m',n')=(k,-l)$ satisfy
\eqref{kaka}. The same is true for the pairs $(-l,k)$
and $(0,0)$. Then the $(k,l)$-flatness of $\gammab$
implies that
\begin{align*}
\int_\cbb |z^{-l}\bar z^k - 1|^2\D\mu(z) =
\gamma_{k-l,k-l} - \gamma_{-l,k} - \gamma_{k,-l} +
\gamma_{0,0} = 0,
\end{align*}
which means that $\supp \mu \subset \zcal_p$ with
$p(z,\bar z)=z^{-l}\bar z^k - 1$. As a consequence,
$\zcal_p \subset \tbb$, which shows that $\mu$
satisfies (ii) (observe that if $z\in\zcal_p$ and
$k+l\neq 0$, then $z\in G_{|k+l|}$).

In the remaining case of $l=0$, which is analogous to
that of $k=0$, $\mu$ satisfies the condition (ii) (in
fact, $\supp \mu \subset G_k$).

To prove the reverse implication, assume that
$\gammab$ is a complex moment sequence with a
representing measure $\mu$ satisfying one the
conditions (i)-(iii). First, we consider the case (i),
that is $\supp \mu \subset \{0\} \cup G_r$ for some
integer $r \Ge 1$. Then
   \begin{align*}
\gamma_{m,n} = c \, \delta_{m+n,0} + \sum_{j=0}^{r-1}
a_j \E^{\frac{2\pi j (m - n) \I}{r}}, \quad (m,n) \in
\nfr,
   \end{align*}
where $c=\mu(\{0\})$ and
$a_j=\mu\big(\big\{\E^{\frac{2\pi j
\I}{r}}\big\}\big)$ for $j=0,\ldots, r-1$. If $r=1$,
then $\gamma_{m,n} = c \, \delta_{m+n,0} + a_0$ for
$(m,n) \in \nfr$, which means that $\gammab$ is
$(1,1)$-flat. This covers the first choice in
\eqref{ptaszki}. It is a matter of routine to verify
that if $r>1$, then $\gammab$ is $(1,r-1)$-flat. This
covers the second choice in \eqref{ptaszki}. It is
easily seen that if (ii) holds, then
   \begin{align*}
\gamma_{m,n}=
   \begin{cases}
\gamma_{m-n,0} & \text{if } m\Ge n,
   \\
\gamma_{0,n-m} & \text{otherwise,}
   \end{cases}
\qquad (m,n)\in \nfr,
   \end{align*}
which implies that $\gamma_{m,n}$ depends on $m-n$.
This covers the third case in \eqref{ptaszki}.
Finally, if (iii) holds, then
$\gamma_{m,n}=\gamma_{m+n,0}$ for all $(m,n)\in \nfr$,
which covers the fourth case in \eqref{ptaszki}. This
completes the proof.
   \end{proof}
   We now proceed to complex moment sequences induced
by Hamburger moment sequences. We say that a sequence
$\sbold=\{s_n\}_{n=0}^\infty \subset \rbb$ is a {\em
Hamburger moment sequence} if there exists a Borel
measure $\tau$ on $\rbb$ such that
   \begin{align} \label{tau}
s_n=\int_\rbb x^n\D\tau(x),\quad n\geqslant 0;
   \end{align}
such a measure is called a {\em representing measure}
for $\sbold$. If $\tau$ in \eqref{tau} is unique, then
$\sbold$ is called a {\em determinate} Hamburger
moment sequence.
   \begin{pro}\label{Ham-c}
Let $\sbold=\{s_n\}_{n=0}^\infty$ be a sequence of
real numbers. Define $\gammab =
\{\gamma_{m,n}\}_{(m,n) \in \nfr}$ by
   \begin{align} \label{mplsn}
\gamma_{m,n} = s_{m+n}, \quad (m,n)\in\nfr.
   \end{align}
Then for $\gammab$ given by \eqref{mplsn} the
following assertions hold{\em :}
   \begin{enumerate}
   \item[(i)] $\sbold$ is a Hamburger moment sequence
if and only if $\gammab$ is a complex moment sequence,
or equivalently, if and only if $\gammab$ is a complex
moment sequence whose every representing measure is
supported in $\rbb$,
   \item[(ii)] $\sbold$ is a determinate Hamburger moment sequence
if and only if $\gammab$ is a determinate complex
moment sequence,
   \item[(iii)] if $\sbold$ is an indeterminate
Hamburger moment sequence and $x_0 \in \rbb$, then
$\gammab$ is an indeterminate complex moment sequence
which has infinitely many representing measures $\mu$
such that $\mu(\{x_0\}) > 0$; moreover, the
cardinality of $\PDE(\gammab)$ is equal to
$\boldsymbol{\mathfrak c}$.
   \end{enumerate}
   \end{pro}
   \begin{proof}
(i) Suppose that $\sbold$ is a Hamburger moment
sequence with a representing measure $\tau$. It is
easily seen that the Borel measure $\mu$ on $\cbb$
defined by
   \begin{align} \label{numerek1}
\mu(\sigma)=\tau(\sigma\cap\rbb), \quad \sigma \in
\borel{\cbb},
   \end{align}
is a representing measure for $\gammab$ supported in
$\rbb$. Suppose now that $\gammab$ is a complex moment
sequence with a representing measure $\varrho$. Since
$\gammab$ is $(1,1)$-flat, Theorem \ref{unif} implies
that $\varrho$ is supported in $\rbb$. As a
consequence, $\sbold$ is a Hamburger moment sequence
with the representing measure $\tau$ defined by
   \begin{align} \label{numerek2}
\tau(\sigma) = \varrho(\sigma), \quad \sigma \in
\borel{\rbb}.
   \end{align}
The above argument concerning the support of $\varrho$
also establishes the second equivalence in (i).

(ii) This is a direct consequence of \eqref{numerek1}
and \eqref{numerek2} and the fact that each
representing measure for $\gammab$ is supported in
$\rbb$ as shown in (i).

(iii) Suppose $\sbold$ is an indeterminate Hamburger
moment sequence. Then, by \cite[Theorem 2.13]{Sh-Tam}
(see also \cite[Theorem 5]{sim}), there exist two
representing measures $\tau_0$ and $\tau_1$ of
$\sbold$ such that $\tau_0(\{x_0\}) =0$ and
$\tau_1(\{x_0\})> 0$. Taking convex combinations
$\tau_{\alpha}\okr\alpha \tau_1 + (1-\alpha)\tau_0$
for $\alpha \in (0,1)$, we get representing measures
for $\sbold$ such that $\tau_{\alpha}(\{x_0\})
> 0$ for all $\alpha \in (0,1]$. It is easily seen that
the corresponding representing measures $\mu_{\alpha}$
of $\gammab$ given by \eqref{numerek1} have the
property that $\mu_{\alpha} (\{x_0\}) > 0$ for $\alpha
\in (0,1]$ and $\mu_{\alpha} \neq \mu_{\beta}$ for all
$\alpha, \beta \in (0,1]$ such that $\alpha
\neq\beta$. The ``moreover'' part follows from
Proposition \ref{snu2}.
   \end{proof}
Regarding Proposition \ref{Ham-c}, one should provide
an example of an indeterminate Hamburger moment
sequence. One of the possible choices is the famous
example $\{\E^{(n+1)^2/4}\}_{n=0}^{\infty}$ due to
Stieltjes (see \cite{stiel}). Below, we present a
class of indeterminate Hamburger moment sequences
introduced recently in \cite{c-s-s}.
   \begin{exa}\label{m+n}
Fix a nonzero complex function $\omega$ on $\rbb$ of
class $\mathcal C^\infty$ whose support is compact and
define the sequence $\{a_n^\omega\}_{n=0}^\infty$ by
   \begin{align*}
a_n^\omega = (-\I)^n \int_\rbb \frac{\D^n \omega}{\D
x^n} (x) \, \overline{\omega(x)} \, \D x, \quad
n=0,1,2,\ldots
   \end{align*}
This is an indeterminate Hamburger moment sequence
with striking properties. Namely, one can find
continuum explicitly described representing measures
for $\{a_n^\omega\}_{n=0}^{\infty}$ such that
   \begin{itemize}
   \item the support of each of them is in arithmetic
progression,
   \item the supports of all these measures together
partition $\rbb$,
   \item all of them are of infinite
order\,\footnote{\;This means that for any such
measure $\tau$, the codimension of polynomials in
$L^2(\tau)$ is infinite.}.
   \end{itemize}
All these three conditions hold under the assumption
that the Fourier transform of $\omega$ does not vanish
on $\rbb$ (such $\omega$ always exists!). Hence, for
any $x_0 \in \rbb$, there exists a representing
measure $\tau$ of $\{a_n^\omega\}_{n=0}^\infty$ in the
above mentioned family such that $\tau(\{x_0\})
> 0$.
   \hfill$\triangleleft$
   \end{exa}
   Instead of Hamburger moment sequences we may
consider Herglotz moment sequences and the complex
moment sequences induced by them. We say that a
sequence $\sbold=\{s_n\}_{n=-\infty}^\infty$ of
complex numbers is a {\em Herglotz $($trigonometric$)$
moment sequence} if there exists a Borel measure
$\varrho$ on $\tbb$ such that
   \begin{align} \label{Herg}
s_n = \int_{\tbb} z^n \D \varrho(z), \quad n=0,\pm 1,
\pm 2, \dots;
   \end{align}
such a measure is called a {\em representing measure}
for $\sbold$. Every Herglotz moment sequence is
determinate, that is the measure $\varrho$ in
\eqref{Herg} is unique (see \cite[Theorem 5.1.2]{Akh};
see also \cite[Theorem 1.7.2]{b-s}). It is worth
mentioning that a Herglotz moment sequence $\sbold$
has the following Hermitian symmetry property:
   \begin{align*}
s_n = \overline{s_{-n}}, \quad n=0,\pm 1, \pm 2,
\dots.
   \end{align*}
This means that such $\sbold$ is uniquely determined
by its entries $s_n$ with $n=0,1,2,\ldots$.
   \begin{pro}
Let $\sbold=\{s_n\}_{n=-\infty}^\infty$ be a sequence
of complex numbers. Define $\gammab =
\{\gamma_{m,n}\}_{(m,n) \in \nfr}$ by
   \begin{align*}
\gamma_{m,n} = s_{m-n}, \quad (m,n)\in\nfr.
   \end{align*}
Then $\sbold$ is a Herglotz moment sequence if and
only if $\gammab$ is a complex moment sequence.
Moreover, if $\sbold$ is a Herglotz moment sequence,
then $\gammab$ is a determinate complex moment
sequence whose unique representing measure is
supported in $\tbb$, $\PDE(\gammab)=\{\Gammab\}$ and
$\Gammab$ is determinate.
   \end{pro}
   \begin{proof}
If $\sbold$ is a Herglotz moment sequence with a
representing measure $\varrho$, then $\gammab$ is a
complex moment sequence with the representing measure
$\mu$ given by
   \begin{align*}
\mu(\sigma)=\varrho(\sigma\cap\tbb),\quad \sigma \in
\borel{\cbb}.
   \end{align*}
Suppose that $\gammab$ is a complex moment sequence
with a representing measure $\nu$. Since $\gammab$ is
$(1,-1)$-flat, Theorem \ref{unif} ensures us that the
support of $\nu$ is contained in $\tbb$. Thus the
Borel measure $\varrho$ on $\tbb$ given by
   \begin{align*}
\varrho(\sigma) = \nu(\sigma), \quad \sigma \in
\borel{\tbb},
   \end{align*}
is a representing measure for $\sbold$. The ``moreover
part follows from the determinacy of Herglotz moment
sequences, the fact that $\nu(\{0\}) = 0$ and
Theorem~\ref{null}.
   \end{proof}
Regarding Proposition \ref{snu2}, we show that it may
happen that a complex moment sequence is indeterminate
and that none of its representing measures satisfies
\eqref{munul}.
   \begin{exa}\label{no-atom}
Consider an indeterminate Hamburger moment sequence
$\sbold = \{s_n\}_{n=0}^{\infty}$ (see e.g., Example
\ref{m+n}). Let $\tau$ be a representing measure for
$\sbold$. Define the Borel measure $\mu$ on $\cbb$ by
   \begin{align} \label{musig}
\mu(\sigma) = \tau((\sigma - \I)\cap \rbb), \quad
\sigma \in \borel{\cbb}.
   \end{align}
Define the complex moment sequence
$\gammab=\{\gamma_{m,n}\}_{(m,n)\in \nfr}$ by
   \begin{align*}
\gamma_{m,n} = \int_\cbb z^m\bar z^n \D\mu(z),\quad
m,n\geqslant 0.
   \end{align*}
It follows from \eqref{musig} and the indeterminacy of
$\sbold$ that $\gammab$ is indeterminate. By
\eqref{musig} again, the measure $\mu$ is supported in
$\rbb+\I$. Since $\rbb+\I=\zcal_p$ with $p(z,\bar z) =
z-\bar z - 2 \I$, we infer from Proposition
\ref{inv_supp} that each representing measure for
$\gammab$ is supported in $\rbb+\I$, and consequently
none of them satisfies \eqref{munul}.
   \hfill$\triangleleft$
   \end{exa}
   \section{Representing measures whose support
is Zariski dense} Our goal in this section is to
establish a wide class of complex moment sequences,
each of which has the following properties:
   \begin{itemize}
   \item[$1^{\circ}$] it is indeterminate,
   \label{Marcypo1}
   \item[$2^{\circ}$] none of its
representing measures has an atom at $0$,
   \item[$3^{\circ}$] all its representing
measures have supports dense in $\cbb$ with respect to
the Zariski topology.
   \end{itemize}
Let us recall that the Zariski topology on $\cbb$
consists of all the sets of the form $\cbb \setminus
\zcal_p$, where $p\in \cbb[z,\bar z]$, and it
satisfies the T$_1$ separation axiom. We refer the
reader to \cite{be-ris,BoCoRo} for more information on
the Zariski topology. The property $3^{\circ}$ means
that the support of any representing measure for the
complex moment sequence under consideration is not
contained in a proper real algebraic set $\zcal_p$, as
opposed to Example~ \ref{no-atom}.

For a technical reason, it is much simpler to state
and prove our result in terms of two-dimensional
Hamburger moment problem; afterwards we will turn back
to complex moment sequences. We say that $\abold=
\{a_{m,n}\}_{m,n=0}^\infty \subset \rbb$ is a {\em
two-dimensional Hamburger moment sequence} if there
exists a Borel measure $\varrho$ on $\rbb^2$, called a
{\em representing measure} for $\abold$, such that
   \begin{align*}
a_{m,n} = \int_{\rbb^2} x^m y^n \D \varrho(x,y), \quad
m,n \Ge 0.
   \end{align*}
If such $\varrho$ is unique, $\abold$ is called {\em
determinate}.

   We begin with recalling a known fact which is
indispensable in this section. For the reader's
convenience, we include its simple proof (see
\cite{Pet} for more on this topic).
   \begin{lem} \label{abol}
If $\abold=\{a_{m,n}\}_{m,n=0}^\infty$ is a
two-dimensional Hamburger moment sequence and
$\varrho$ its representing measure, then
$\{a_{m,0}\}_{m=0}^\infty$ $($resp.,
$\{a_{0,n}\}_{n=0}^\infty$$)$ is~ a~ Hamburger moment
sequence with the representing measure $\varrho \circ
\pi_1^{-1}$ $($resp., $\varrho \circ
\pi_2^{-1}$$)$,~and
   \begin{align} \label{marcy1}
\supp \varrho \circ \pi_1^{-1} = \overline{\pi_1(\supp
\varrho)} \quad \Big(\text{resp., }\supp \varrho \circ
\pi_2^{-1} = \overline{\pi_2(\supp \varrho)}\Big),
   \end{align}
where $\pi_j\colon \rbb^2 \to \rbb$, $j=1,2$, are
mappings given by
   \begin{align} \label{abold}
\text{$\pi_1(x,y)=x$ and $\pi_2(x,y)=y$ for $(x,y)\in
\rbb^2$.}
   \end{align}
   \end{lem}
   \begin{proof}
The equalities in \eqref{marcy1} follow from
\cite[Lemma 3.2]{St-St}. Using the measure transport
theorem, we get
   \begin{align*}
a_{m,0} = \int_{\rbb^2} x^m\D\varrho(x,y) = \int_\rbb
x^m \D\varrho\circ \pi_1^{-1}(x), \quad m \Ge 0.
   \end{align*}
A similar argument applies to
$\{a_{0,n}\}_{n=0}^\infty$.
   \end{proof}
The following shows that in some cases supports of
representing measures can be localized in
non-algebraic subsets of $\cbb$ (cf.\ Proposition
\ref{inv_supp}; see also Proposition~\ref{fontan2}).
   \begin{pro} \label{fontan}
Let $\abold=\{a_{m,n}\}_{m,n=0}^\infty$ be a
two-dimensional Hamburger moment sequence and
$\varrho$ its representing measure. Then the following
assertions hold{\em :}
   \begin{enumerate}
   \item[(i)] if the set $\pi_1(\supp \varrho)$ $($resp.,
$\pi_2(\supp \varrho)$$)$ is bounded, then
$\{a_{m,0}\}_{m=0}^\infty$ $($resp.,
$\{a_{0,n}\}_{n=0}^\infty$$)$ is a determinate
Hamburger moment sequence,
   \item[(ii)] if $\{a_{m,0}\}_{m=0}^\infty$ $($resp.,
$\{a_{0,n}\}_{n=0}^\infty$$)$ is a determinate
Hamburger moment sequence, then for any representing
measure $\tilde\varrho$ for~ $\abold,$
   \begin{gather}  \label{smiechDC2}
\supp \tilde \varrho \subset \overline{\pi_1(\supp
\varrho)} \times\rbb \quad \Big(\text{resp., }\supp
\tilde \varrho \subset \rbb \times
\overline{\pi_2(\supp \varrho)}\Big),
   \\ \label{smiechDC}
\overline{\pi_1(\supp \tilde\varrho)} =
\overline{\pi_1(\supp \varrho)} \quad
\Big(\text{resp., } \overline{\pi_2(\supp
\tilde\varrho)} = \overline{\pi_2(\supp \varrho)}
\Big).
   \end{gather}
   \end{enumerate}
   \end{pro}
   \begin{proof}
By symmetry, it suffices to consider the case of
$\{a_{m,0}\}_{m=0}^\infty$.

(i) If $\pi_1(\supp \varrho)$ is bounded, then by
Lemma \ref{abol} and \eqref{marcy1} the Hamburger
moment sequence $\{a_{m,0}\}_{m=0}^\infty$ has a
compactly supported representing measure and as such
is determinate (see \cite[p.\ 50]{Fug}).

(ii) Suppose $\{a_{m,0}\}_{m=0}^\infty$ is a
determinate Hamburger moment sequence and
$\tilde\varrho$ is a representing measure for
$\abold$. By Lemma~ \ref{abol}, $\tilde\varrho \circ
\pi_1^{-1}=\varrho \circ \pi_1^{-1}$. This together
with \eqref{marcy1} implies the first equality in
\eqref{smiechDC}. Set $S=\supp \varrho \circ
\pi_1^{-1}$. Then the equality $\tilde\varrho \circ
\pi_1^{-1}=\varrho \circ \pi_1^{-1}$ yields
   \begin{align*}
0=\tilde \varrho \circ \pi_1^{-1}(\rbb\setminus S) =
\tilde \varrho (\rbb^2\setminus (S \times \rbb)).
   \end{align*}
Since $\rbb^2\setminus (S \times \rbb)$ is an open
subset of $\rbb^2$, we see that $\supp \tilde \varrho
\subset S\times \rbb$. Combined with \eqref{marcy1},
this gives the first inclusion in \eqref{smiechDC2}.
   \end{proof}
   We now turn to the main result of this section.
   \begin{thm} \label{DC1}
Let $\sbold=\{s_m\}_{m=0}^{\infty}$ be a determinate
Hamburger moment sequence such that its unique
representing measure $\mu$ has infinite support and
$\mu(\nul)=0$, and let $\tbold=\{t_n\}_{n=0}^{\infty}$
be an indeterminate Hamburger moment sequence. Then
the sequence $\sbold\otimes \tbold =
\{(\sbold\otimes\tbold)_{m,n}\}_{m,n=0}^\infty$
defined by
   \begin{align*}
(\sbold\otimes\tbold)_{m,n} = s_m t_n, \quad m,n \Ge
0,
   \end{align*}
is a two-dimensional Hamburger moment sequence
satisfying the conditions $1^{\circ}$, $2^{\circ}$ and
$3^{\circ}$ of page {\em \pageref{Marcypo1}} with
$\rbb^2$ in place of $\cbb$.
   \end{thm}
   \begin{proof}
With no loss of generality, we can assume that
$t_0=1$. The indeterminacy of $\tbold$ implies that
   \begin{align} \label{supinf}
\text{the support of any representing measure for
$\tbold$ is infinite.}
   \end{align}
It follows from the Fubini theorem that the mapping
$\nu\mapsto \mu\otimes \nu$ acts between the set of
all representing measures for $\tbold$ and the set of
all representing measures for $\sbold\otimes\tbold$,
where $\mu\otimes \nu$ stands for the product measure
of $\mu$ and $\nu$. Since $\mu(\rbb) \neq 0$, the
mapping is easily seen to be injective. Hence, by the
indeterminacy of $\tbold$, the sequence
$\sbold\otimes\tbold$ is indeterminate. This shows
$1^{\circ}$.

For $2^{\circ}$, take any representing measure
$\varrho$ for $\abold \okr \sbold \otimes \tbold$.
Since $a_{m,0} = s_m$ for all integers $m\Ge 0$ and
$\sbold$ is determinate, we infer from Lemma
\ref{abol} that $\mu=\varrho\circ \pi_1^{-1}$, which
implies that
   \begin{align*}
0=\mu(\{0\}) = \varrho(\{0\}\times \rbb).
   \end{align*}
This yields $2^{\circ}$.

It remains to prove $3^{\circ}$. Take any representing
measure $\nu$ for $\tbold$. It is a well-known and
easy to prove fact that
   \begin{align} \label{spppro}
\supp\mu\otimes\nu=\supp\mu\times\supp \nu.
   \end{align}
Suppose that a polynomial $p\in \cbb[x,y]$ vanishes on
$\supp\mu\otimes\nu$. We will show that $p=0$. Indeed,
if $x \in \supp\mu$, then by \eqref{supinf} and
\eqref{spppro}, the polynomial $y\mapsto p(x,y)$
vanishes on an infinite subset of $\rbb$, and
consequently $p(x,y)=0$ for all $x\in \supp \mu$ and
$y\in \rbb$. Hence, for every $y\in \rbb$, the
polynomial $x \mapsto p(x,y)$ vanishes on an infinite
subset of $\rbb$, and thus $p(x,y)=0$ for all $x\in
\rbb$. As a consequence, $p=0$. Suppose, contrary to
$3^{\circ}$, that there exists a representing measure
$\varrho$ for $\sbold\otimes\tbold$ such the Zariski
closure of $\supp \varrho$ is a proper subset of
$\rbb^2$. This means that there exists a nonzero
polynomial $q\in \cbb[x,y]$ such that the measure
$\varrho$ is supported in the real algebraic set
$S=\{(x,y)\in \rbb^2\colon q(x,y)=0\}$. Since
Proposition \ref{inv_supp} remains valid for the
two-dimensional Hamburger moment problem, we infer
that any representing measure for
$\sbold\otimes\tbold$ must be supported in $S$. In
particular, this should hold for $\mu\otimes \nu$,
which is a contradiction. This yields $3^{\circ}$ and
completes the proof.
   \end{proof}
Now, we turn back to the complex case. Recall that
there is a one-to-one correspondence between the set
of all two-dimensional Hamburger moment sequences
$\abold=\{a_{m,n}\}_{m,n=0}^\infty$ and the set of all
complex moment sequences
$\gammab=\{\gamma_{m,n}\}_{m,n=0}^\infty$ given by
   \begin{align} \label{baba}
\gamma_{m,n} = \int_{\rbb^2} (x+\I y)^m (x-\I y)^n \D
\varrho(x,y) = \sum_{k,l\Ge 0} \alpha_{k,l}^{m,n}
a_{k,l}, \quad (m,n)\in \nfr,
   \end{align}
where $\varrho$ is a representing measure for $\abold$
and $\{\alpha_{k,l}^{m,n}\}_{k,l=0}^{\infty}$,
$(m,n)\in \nfr$, are systems of complex numbers (each
with finitely many nonzero entries) uniquely
determined by the equations
   \begin{align*}
(x+\I y)^m (x-\I y)^n = \sum_{k,l\Ge 0}
\alpha_{k,l}^{m,n} x^ky^l, \quad x,y\in \rbb.
   \end{align*}
Moreover, the sets of all representing measures for
$\abold$ and the corresponding $\gammab$ coincide. We
refer the reader to \cite[Appendix A]{c-s-s-rh} for
more details. As a consequence, if $\abold =
\sbold\otimes\tbold$, where $\sbold$ and $\tbold$ are
as in Theorem \ref{DC1}, then the corresponding
$\gammab$ satisfies the conditions $1^{\circ}$,
$2^{\circ}$ and $3^{\circ}$ on page
\pageref{Marcypo1}.
   \section{Representing measures on real algebraic sets}
In contrast to the previous section, we will now focus
on the complex moment sequences having representing
measures on real algebraic sets that are different
from~ $\cbb$. The most satisfactory result establishes
a one-to-one correspondence between representing
measures for a complex moment sequence and its
positive definite extensions on $\nfr_+$ in the case
when the mapping $\cbb^* \ni z \mapsto \frac{z}{\bar
z} \in \tbb$ restricted to the algebraic set in
question is injective (see Theorem \ref{0notatom}(v)).
In Example \ref{wiele} we gather some classes of real
algebraic sets meeting this requirement. After
examining the Witch of Agnesi (one of these examples)
we conclude that there is no complex moment
counterpart of the partitioning property of the family
of N-extremal measures as in one-dimensional Hamburger
moment problem (see Proposition \ref{fontan3} and the
discussion preceding it).

It is a well-known fact that a mapping $f\colon X\to
Y$ between nonempty sets $X$ and $Y$ is injective if
and only if the mapping $2^Y \ni \sigma \longmapsto
f^{-1}(\sigma) \in 2^X$ is surjective. Following this,
we say that a Borel mapping $f\colon X \to Y$ between
topological Hausdorff spaces $X$ and $Y$ (i.e., a
mapping such that $f^{-1}(\sigma) \in \borel{X}$ for
all $\sigma \in \borel{Y}$) is {\em Borel injective}
if the related inverse image mapping $\borel{Y} \ni
\sigma \longmapsto f^{-1}(\sigma) \in \borel{X}$ is
surjective. Below, we show that the notions of
injectivity and Borel injectivity coincide for
continuous mappings $f$ defined on $\sigma$-compact
topological Hausdorff spaces.
   \begin{pro} \label{Binj}
Let $f\colon X \to Y$ be a mapping between topological
Hausdorff spaces $X$ and $Y$. Then the following
assertions hold{\em :}
   \begin{enumerate}
   \item[(i)] if $f$ is Borel injective, then $f$
is injective,
   \item[(ii)] if $f$ is
continuous and $X$ is $\sigma$-compact, then $f$ is
Borel injective if and only if it is injective.
   \end{enumerate}
   \end{pro}
   \begin{proof}
(i) Suppose that, contrary to our claim, $f(x_1) =
f(x_2)$ for some distinct points $x_1$ and $x_2$ of
$X$. Then the singleton $\{x_1\}$ is closed and so
there exists $\sigma'\in \borel{Y}$ such that
$\{x_1\}=f^{-1}(\sigma')$. However, $x_2 \in
f^{-1}(\sigma')$, which is a contradiction.

(ii) Suppose $f$ is continuous and $X$ is
$\sigma$-compact. Clearly, $f$ is a Borel mapping. In
view of (i), it suffices to prove the ``if'' part.
Assume that $f$ is injective. First, we show that
   \begin{align} \label{closim}
\text{if $F$ is a closed subset of $X$, then $f(F)\in
\borel{Y}$.}
   \end{align}
Indeed, by $\sigma$-compactness of $X$, there exists a
sequence $\{K_n\}_{n=1}^{\infty}$ of compact subsets
of $X$ such that $X=\bigcup_{n=1}^{\infty} K_n$. Since
each $K_n\cap F$ is compact and consequently, by the
continuity of $f$, each $f(K_n\cap F)$ is compact, we
see that
   \begin{align*}
f(F) = f\bigg(\bigcup_{n=1}^{\infty} K_n\cap F\bigg) =
\bigcup_{n=1}^{\infty} f(K_n\cap F) \in \borel{Y},
   \end{align*}
which completes the proof of \eqref{closim}. Set
   \begin{align*}
\mathcal{A}_{f} = \{\sigma \in \borel{X}\colon
f(\sigma) \in \borel{Y}\}.
   \end{align*}
It follows from \eqref{closim} that $X\in
\mathcal{A}_{f}$. Since $f$ is injective and $f(X)\in
\borel{Y}$, we deduce that $X\setminus \sigma \in
\mathcal{A}_{f}$ whenever $\sigma \in
\mathcal{A}_{f}$. Clearly, $\bigcup_{n=1}^{\infty}
\sigma_n \in \mathcal{A}_{f}$ whenever
$\{\sigma_n\}_{n=1}^{\infty} \subset \mathcal{A}_{f}$.
This means that $\mathcal{A}_{f}$ is a
$\sigma$-subalgebra of $\borel{X}$, which, by
\eqref{closim}, contains all the open subsets of $X$.
Hence, $\mathcal{A}_{f}=\borel{X}$, that is
$f(\sigma)\in \borel{Y}$ for every $\sigma \in
\borel{X}$. To prove the Borel injectivity of $f$,
take $\sigma \in \borel{X}$. Since
$\mathcal{A}_{f}=\borel{X}$, we see that $\sigma'\okr
f(\sigma) \in \borel{Y}$. By the injectivity of $f$,
we deduce that $\sigma = f^{-1}(\sigma')$, which
completes the proof of Borel injectivity of~ $f$.
   \end{proof}
   \begin{rem}
   Note that in general injective (or even bijective)
continuous mappings between topological Hausdorff
spaces may not be Borel injective. Indeed, the mapping
$f\colon X \to Y$, where $X$ is the real line equipped
with the discrete topology and $Y$ is the real line
equipped with the Euclidean topology, defined by
$f(x)=x$ for $x\in X$, is bijective and continuous,
but not Borel injective because $\borel{Y}
\varsubsetneq 2^{X}$ (see \cite[Remarks 2.21]{Rud}).
   \hfill$\triangleleft$
   \end{rem}
In this section we will focus on restrictions of the
continuous mapping
   \begin{align} \label{odw-2}
\psi\colon \cbb^* \to \tbb, \quad \psi(z) =
\frac{z}{\bar z} = \bigg(\frac{z}{|z|}\bigg)^2, \quad
z\in \cbb^*.
   \end{align}
For a given Borel measure $\mu$ on $\cbb$ and a
nonempty Borel subset $Z$ of $\cbb^*$, we denote by
$\mu\circ(\psi|_Z)^{-1}$ the transport of the Borel
measure $\mu|_{\borel{Z}}$ via $\psi|_Z\colon Z \to
\tbb$ given by
   \begin{align} \label{dodo2}
(\mu\circ(\psi|_Z)^{-1})(\sigma) \okr
\mu((\psi|_Z)^{-1}(\sigma)), \quad \sigma\in
\borel{\tbb},
   \end{align}
In the context of restrictions of $\psi$, Proposition
\ref{Binj} can be specified as follows.
   \begin{cor} \label{Binj-c}
If $Z$ is a $\sigma$-compact subset of $\cbb$ such
that $0\notin Z$, then the mapping $\psi|_Z$ is Borel
injective if and only if it is injective.
   \end{cor}
Proposition \ref{Binj} enables us to formulate a
geometric criterion for Borel injectivity of
restrictions of $\psi$.
   \begin{pro} \label{lines}
Let $Z$ be a nonempty subset of $\cbb^*$. Then the
mapping $\psi|_Z$ is injective if and only if the
intersection of $Z$ and any straight line passing
through the origin contains at most one point.
   \end{pro}
   \begin{proof}
Suppose the intersection of $Z$ and any straight line
passing through the origin contains at most one point.
Take $z_1,z_2\in Z$ such that $\psi(z_1)=\psi(z_2)$.
Note that there exist $t_1, t_2\in\rbb$, such that
$|t_1-t_2| \Le \pi$, $z_1=|z_1|\E^{\I t_1}$ and
$z_2=|z_2|\E^{\I t_2}$. Since $Z\subset \cbb^*$, we
see that $\E^{2\I t_1} = \E^{2\I t_2}$, which gives
$t_1=t_2$ or $|t_1-t_2|=\pi$. In both cases $z_1$ and
$z_2$ are points of a straight line passing through
the origin, so, by our assumption, $z_1=z_2$. This
proves the injectivity of $\psi|_Z$. The converse
implication follows easily from the fact that $\psi$
is constant on any straight line $\{r\E^{\I t}\colon r
\in \rbb\}$ intersected with $\cbb^*$, where $t\in
[0,\pi)$. This completes the proof.
   \end{proof}
   \begin{cor}
Let $Z\subset \cbb^*$ be a nonempty set. Suppose
$\varDelta$ is a subset of $\rbb$ such that $|t_1-t_2|
< \pi$ for all $t_1,t_2\in \varDelta$, and $r\colon
\varDelta \to (0,\infty)$ is a function for which the
mapping $\phi\colon \varDelta \ni t \longmapsto
r(t)\E^{\I t}\in Z$ is surjective. Then $\phi$ and
$\psi|_Z$ are injective.
   \end{cor}
   Before formulating the main result of this section,
we state a crucial lemma which is in the spirit of
quasi-determinacy (cf.\ Theorem \ref{semi}).
   \begin{lem} \label{chmura2}
If $\gammab\colon \nfr \to \cbb$ is a complex moment
sequence and $(\mu_1,\nu_1)$ and $(\mu_2,\nu_2)$ are
representing pairs for some $\Gammab \in
\PDE(\gammab)$, then
   \begin{align} \label{chmur}
\mu_1\circ\psi^{-1} + \nu_1\circ\varphi^{-1} =
\mu_2\circ\psi^{-1} + \nu_2\circ\varphi^{-1},
   \end{align}
where $\varphi$ and $\psi$ are given by \eqref{odw}
and \eqref{odw-2}, respectively, whereas
$\nu_j\circ\varphi^{-1}$ and $\mu_j\circ\psi^{-1}$ are
Borel measures on $\tbb$ given by \eqref{dodo1} and
\eqref{dodo2}, respectively.
   \end{lem}
   \begin{proof}
It follows from the measure transport theorem that
\allowdisplaybreaks
   \begin{align*}
\int_{\cbb^*} z^m \bar z^{-m} \D\mu_j(z) &
\overset{\eqref{odw-2}}= \int_{\cbb^*} \psi(z)^{m}
\D\mu_j(z)
   \\
& = \int_{\tbb} z^{m} \D(\mu_j\circ \psi^{-1})(z),
\quad m = 0, \pm 1,\pm 2, \ldots, \, j=1,2,
   \end{align*}
and similarly
   \begin{align*}
\int_{\tbb} z^m \bar z^{-m} \D\nu_j(z) = \int_{\tbb}
z^{m} \D(\nu_j\circ \varphi^{-1})(z), \quad m = 0, \pm
1,\pm 2, \ldots, \, j=1,2,
   \end{align*}
which implies that
   \begin{align*}
\varGamma_{m,-m} \overset{ \eqref{repr1}} =
\int_{\tbb} z^m \D(\mu_j\circ \psi^{-1} + \nu_j\circ
\varphi^{-1})(z), \quad m = 0, \pm 1,\pm 2, \ldots, \,
j=1,2.
   \end{align*}
Hence, by the determinacy of the Herglotz moment
problem (see Section \ref{Sec3}), the condition
\eqref{chmur} holds. This completes the proof.
   \end{proof}
In what follows:
   \begin{itemize}
   \item $\psi_p\okr \psi|_{\zcal_p}$ whenever $p\in
\cbb[z,\bar z]$ is such that $\zcal_p \neq
\varnothing$ and $0\notin \zcal_p$,
   \item $\mathcal M(\gammab)$ stands for the set
of all representing measures for a complex moment
sequence $\gammab$ on $\nfr$.
   \end{itemize}
We are now in a position to prove the main result of
this section.
   \begin{thm} \label{0notatom} Let
$p\in\cbb[z,\bar z]$ be a polynomial such that
$\zcal_p \neq \varnothing$ and $0\notin \zcal_p$.
Assume that $\gammab\colon \nfr \to \cbb$ is a complex
moment sequence which has a representing measure
supported in $\zcal_p$. Then the following assertions
hold{\em :}
   \begin{enumerate}
   \item[(i)] if $\mu\in \mathcal M(\gammab)$, then
$\supp\mu\subset \zcal_p$,
$\Gammab(\mu)=\{\varGamma_{m,n}(\mu)\}_{(m,n)\in
\nfr_+} \in \PDE(\gammab)$ and $(\mu,0)$ is a
representing pair for $\Gammab(\mu)$, where
   \begin{align*}
\varGamma_{m,n}(\mu) = \int_{\cbb^*} z^m\bar z^n
\D\mu(z), \quad (m,n)\in \nfr_+,
   \end{align*}
   \item[(ii)] if $(\mu,\nu)$ is a representing pair for
some $\Gammab\in \PDE(\gammab)$, then $\mu \in
\mathcal M(\gammab)$ and $\nu=0$,
   \item[(iii)] if $(\mu_1,0)$ and $(\mu_2,0)$ are
representing pairs for some $\Gammab\in
\PDE(\gammab)$, then
   \begin{align} \label{ajjaj}
\mu_1\circ\psi_p^{-1} = \mu_2\circ\psi_p^{-1},
   \end{align}
   \item[(iv)] if $Z$ is a nonempty closed subset of $\zcal_p$
such that $\psi|_Z$ is injective, then the mapping
$\mathcal M_Z(\gammab) \ni \mu \longmapsto
\Gammab(\mu)\in \PDE(\gammab)$ is injective, where
$\mathcal M_Z(\gammab)=\{\mu\in \mathcal M(\gammab)
\colon \supp \mu \subset Z\}$,
   \item[(v)] if $\psi_p$ is injective, then
   \begin{itemize}
   \item the mapping $\mathcal M(\gammab) \ni \mu \longmapsto
\Gammab(\mu)\in \PDE(\gammab)$ is bijective,
   \item every $\Gammab \in \PDE(\gammab)$ is determinate,
   \item $\PDE(\gammab)$ is of cardinality continuum whenever
$\gammab$ is indeterminate.
   \end{itemize}
   \end{enumerate}
   \end{thm}
   \begin{proof}
(i) Suppose $\mu\in \mathcal M(\gammab)$. By
Proposition \ref{inv_supp}, $\mu$ is supported in
$\zcal_p$. Since, by our assumption $0\notin \zcal_p$,
we deduce that $\Gammab(\mu) \in \PDE(\gammab)$ and
$(\mu,0)$ is a representing pair for $\Gammab(\mu)$
(see Lemma \ref{repr2}).

(ii) Assume that $(\mu,\nu)$ is a representing pair
for some $\Gammab\in \PDE(\gammab)$. It follows from
Lemma \ref{repr2} that $\mu + \nu(\mathbb T)\delta_0
\in \mathcal M(\gammab)$, and consequently by (i) and
the assumption that $0\notin \zcal_p$, we have
   \begin{align*}
\nu(\mathbb T) = (\mu + \nu(\mathbb T)\delta_0)(\{0\})
= 0.
   \end{align*}
This implies that $\nu=0$ and thus, by \eqref{repr1},
$\mu \in \mathcal M(\gammab)$.

   (iii) Assume that $(\mu_1,0)$ and $(\mu_2,0)$ are
representing pairs for some $\Gammab\in
\PDE(\gammab)$. Then, by Lemma \ref{chmura2},
$\mu_1\circ\psi^{-1} = \mu_2\circ\psi^{-1}$. By (i)
and (ii), the measures $\mu_1$ and $\mu_2$ are
supported in $\zcal_p$. Therefore, $\mu_j\circ
\psi^{-1} = \mu_j\circ \psi_p^{-1}$ for $j=1,2$, which
yields \eqref{ajjaj}.

(iv) Assume that $Z$ is a nonempty closed subset of
$\zcal_p$ such that $\psi|_Z$ is injective. Since $Z$
is $\sigma$-compact and $0\notin Z$, we infer from
Corollary \ref{Binj-c} that the mapping $\psi|_Z$ is
Borel injective. First note that by (i),
$\Gammab(\mu)\in \PDE(\gammab)$ for every $\mu \in
\mathcal M(\gammab)$. If $\mu_1,\mu_2 \in \mathcal
M_Z(\gammab)$ are such that $\Gammab(\mu_1) =
\Gammab(\mu_2)$, then by (i) and (iii),
$\mu_1\circ\psi_p^{-1} = \mu_2\circ\psi_p^{-1}$. Since
the measures $\mu_1$ and $\mu_2$ are supported in $Z$,
we deduce that $\mu_j\circ \psi_p^{-1} = \mu_j\circ
(\psi|_Z)^{-1}$ for $j=1,2$. Hence,
$\mu_1\circ(\psi|_Z)^{-1} = \mu_2\circ(\psi|_Z)^{-1}$,
which by the Borel injectivity of $\psi|_Z$ leads to
$\mu_1=\mu_2$. As a consequence, the mapping $\mathcal
M_Z(\gammab) \ni \mu \longmapsto \Gammab(\mu)\in
\PDE(\gammab)$ is injective.

(v) Assume that $\psi_p$ is injective. It follows from
(i) and (iv) that the mapping $\mathcal M(\gammab) \ni
\mu \longmapsto \Gammab(\mu)\in \PDE(\gammab)$ is
injective. On the other hand, by (ii), the second term
in \eqref{repr1} must be zero whenever $\Gammab\in
\PDE(\gammab)$ and $(\mu,\nu)$ is a representing pair
for $\Gammab$. This yields the surjectivity of the
mapping $\mathcal M(\gammab) \ni \mu \longmapsto
\Gammab(\mu)\in \PDE(\gammab)$. Therefore, it is a
bijection. This together with (ii) implies that every
$\Gammab \in \PDE(\gammab)$ is determinate. To prove
the last assertion in (v) assume that $\gammab$ is
indeterminate. Since the set $\mathcal M(\gammab)$ is
a convex subset of the set of all Borel measures on
$\cbb$ and $\mathcal M(\gammab)$ is not a one point
set, we see that the cardinality of $\mathcal
M(\gammab)$ is at least continuum. Combined with the
injectivity of $\mathcal M(\gammab) \ni \mu
\longmapsto \Gammab(\mu)\in \PDE(\gammab)$, this
implies that the set $\PDE(\gammab)$ is of cardinality
at least $\boldsymbol{\mathfrak c}$. Since
$\PDE(\gammab) \subset \cbb^{\nfr_+}$ and the
cardinality of $\cbb^{\nfr_+}$ is
$\boldsymbol{\mathfrak c}$, the proof is complete.
   \end{proof}
   \begin{cor} \label{nullab}
Let $p\in\cbb[z,\bar z]$ be a polynomial such that
$\zcal_p\neq \varnothing$, $0 \notin \zcal_p$ and
$\psi_p$ is injective. Suppose $\gammab\colon \nfr \to
\cbb$ is a complex moment sequence having a
representing measure supported in $\zcal_p$. If
$\PDE(\gammab)=\{\Gammab\}$, then $\gammab$ is
determinate.
   \end{cor}
Regarding the assertion (v) of Theorem \ref{0notatom},
it is advisable to know for which polynomials $p$, the
mapping $\psi_p$ is injective. It is easily seen that
the injectivity property of $\psi_p$ fails to hold for
most plane algebraic curves, including circles,
ellipses, hyperbolas, parabolas, lemniscates, etc.
However, we may indicate several polynomials $p$ for
which $\psi_p$ is injective. For convenience, in
Example \ref{wiele} below we use the two real variable
description of real algebraic sets.
   \begin{exa} \label{wiele}
The mapping $\psi_p$ is injective in any of the
following cases:
   \begin{itemize}
   \item[\ding{192}] $p(x,y) = a x + b y - c$, where $a,b,c \in \rbb$,
$a^2 + b^2 > 0$ and $c\neq 0$ (a straight line which
does not contain the origin);
   \item[\ding{193}] $p(x,y) = (y-y_0)^l - a x^{2k}$, where
$k$ is a nonnegative integer, $l>2k$ is an odd
integer, $a>0$ and $y_0 > 0$ (for $k=1$ and $l=3$ this
is a shifted {\em Neil's semicubical parabola}, cf.\
\cite[p.\ 93]{B-S-M-M}, \cite[p.\ 5]{Kunz});
   \item[\ding{194}] $p(x,y) = y(x^2+a) -
b$, where $a,b>0$ (a generalized {\em Witch of
Agnesi}, cf.\ \cite[p.\ 94]{B-S-M-M});
   \item[\ding{195}]  $p(x,y) = ((x-x_0)^2 + y^2)(x-x_0) - 2ay^2$,
where $a>0$ and $x_0 > 0$ (a shifted {\em cissoid of
Diocles}, cf.\ \cite[p.\ 95]{B-S-M-M}, \cite[p.\
5]{Kunz});
   \item[\ding{196}] $p(x,y) = y^l x^{2k} - a$, where $k$ is a
nonnegative integer, $l$ is an odd positive integer
and $a \in \rbb\setminus \{0\}$.
   \end{itemize}
The injectivity of $\psi_p$ will be deduced from
Proposition \ref{lines} by verifying that the
intersection of $\zcal_p$ and any straight line
passing through the origin contains at most one point.

The case \ding{192} is obvious. Let $p$ be as in
\ding{193}. Then $\zcal_p$ is located in the upper
half-plane. The case of the line $x=0$ is plain. Since
the set $\zcal_p \cap \{z\in \cbb\colon \RE(z)\Ge 0\}$
is the graph of a strictly increasing concave function
on the interval $[0,\infty)$ whose value at $0$ is
positive, it intersects the line $y=cx$ in exactly one
point whenever $c>0$. The case $c<0$ follows by the
symmetry of $\zcal_p$ with respect to the reflection
across the line $x=0$.

Suppose $p$ is as in \ding{194}. Then $\zcal_p$ is
contained in the upper half-plane. The case of the
line $x=0$ is trivial. Since the set $\zcal_p \cap
\{z\in \cbb\colon \RE(z)\Ge 0\}$ is the graph of a
strictly decreasing positive function on the interval
$[0,\infty)$, it intersects the line $y=cx$ in exactly
one point whenever $c>0$. As above, the case $c<0$
follows by the symmetry of $\zcal_p$ with respect to
the reflection across the line $x=0$.

Let $p$ be as in \ding{195}. This time $\zcal_p$ is a
subset of the right half-plane. Again, the case of the
line $y=0$ is obvious. Since the set $\zcal_p \cap
\{z\in \cbb\colon \IM(z)\Ge 0\}$ is the graph of a
strictly increasing convex function on the interval
$[x_0,x_0+2a)$ that vanishes at $x_0$, it intersects
the line $y=cx$ in exactly one point whenever $c>0$.
The case $c<0$ follows by the symmetry of $\zcal_p$
with respect to the reflection across the line $y=0$.

Finally, the case when $p$ is as in \ding{196} is
straightforward.
   \hfill$\triangleleft$
   \end{exa}
Regarding Theorem \ref{0notatom}(iv), we note that
though in general the mapping $\psi$ given by
\eqref{odw-2} is not injective on plane algebraic
curves, it becomes such on appropriately chosen parts
of them, e.g., one branch of a hyperbola, an arc of a
parabola, etc. In turn, Proposition \ref{inv_supp}
which helps to localize the supports of representing
measures of a complex moment sequence on a real
algebraic set can be enforced with the help of
Proposition \ref{fontan} as follows (cf.\ Remark
\ref{reM-2}).
   \begin{pro}\label{fontan2}
If $\gammab=\{\gamma_{m,n}\}_{(m,n)\in \nfr}$ is a
complex moment sequence which has a representing
measure $\mu$ supported in a real algebraic set $Z$
such that the set $\pi_1(\supp\mu)$ $($resp.,
$\pi_2(\supp\mu)$$)$ is bounded, then
   \begin{align*}
\supp \tilde \mu \subset Z \cap
\big(\overline{\pi_1(\supp\mu)}\times \rbb\big) \quad
\Big(\text{resp., } \supp \tilde \mu \subset Z \cap
\big(\rbb \times \overline{\pi_2(\supp\mu)}\big)\Big)
   \end{align*}
for any representing measure $\tilde \mu$ for
$\gammab$, where $\pi_1$ and $\pi_2$ are as in
\eqref{abold}.
   \end{pro}
Proposition \ref{fontan2} can be applied e.g.\ to the
Witch of Agnesi (see Example \ref{wiele}\ding{194}).
We will show that for such a plane algebraic curve
there is no analogue of an N-extremal measure in the
following sense. Recall that a representing measure
$\mu$ of an indeterminate Hamburger moment sequence is
said to be {\em N-extremal} if complex polynomials are
dense in $L^2(\mu)$. The supports of N-extremal
measures of an indeterminate Hamburger moment sequence
have remarkable properties, namely they are infinite,
have no accumulation points in $\rbb$ and form a
partition of $\rbb$ (see \cite[Theorem 2.13]{Sh-Tam};
see also \cite{sim}). As shown in Proposition
\ref{fontan3} below, this is no longer true for
supports of representing measures of a complex moment
sequence provided at least one of them is contained in
the Witch of Agnesi and has no accumulation point
therein. An analogue of Proposition \ref{fontan3} can
also be formulated and proved for a shifted cissoid of
Diocles (see Example \ref{wiele}\ding{195}). We leave
the details to the reader.

Below, for brevity, we write $x_n \nearrow \infty$
(resp., $x_n \searrow 0$) if $\{x_n\}_{n=1}^{\infty}$
is a strictly increasing (resp., strictly decreasing)
sequence in $\rbb$ which converges to $\infty$ (resp.,
~$0$).
   \begin{pro} \label{fontan3}
Let $\gammab=\{\gamma_{m,n}\}_{(m,n)\in \nfr}$ be a
complex moment sequence with a representing measure
$\mu$ supported in $\zcal_p$, where $p$ is as in
Example {\em \ref{wiele}}\ding{194}. Assume that
$\supp \mu$ is infinite and has no accumulation points
in $\zcal_p$. Then there exists a sequence
$\{z_n\}_{n=1}^{\infty} \subset \zcal_p$ such that
   \begin{enumerate}
   \item[(i)] $0 \Le \RE(z_n) \nearrow \infty$,
   \item[(ii)] for any representing
measure $\tilde \mu$ of $\gammab$,
   \begin{gather} \label{ne5}
\tilde \mu(\{z_n, - \bar z_n\}) = \mu(\{z_n,- \bar
z_n\}) > 0, \quad n\Ge 1,
   \\ \label{ne6}
\supp \tilde \mu \subset \{z_n\colon n\Ge 1\} \cup \{-
\bar z_n\colon n\Ge 1\}.
   \end{gather}
   \end{enumerate}
   \end{pro}
   \begin{proof}
Since $\supp \mu$ is infinite and has no accumulation
points in $\zcal_p$, one can show that there exists a
sequence $\{z_n\}_{n=1}^{\infty} \subset \zcal_p$ such
that
   \begin{gather} \label{ne1}
0 \Le \RE(z_n) \nearrow \infty,
   \\ \label{ne2}
\mu(\{z_n,- \bar z_n\}) > 0, \quad n\Ge 1,
   \\ \label{ne3}
\supp \mu \subset \{z_n\colon n\Ge 1\} \cup \{- \bar
z_n\colon n\Ge 1\}.
   \end{gather}
In fact, the set $\{z_n\colon n\Ge 1\} \cup \{- \bar
z_n\colon n\Ge 1\}$ is the smallest subset of
$\zcal_p$ that contains $\supp \mu$ and is symmetric
with respect to the reflection across the line $x=0$.
It follows from \eqref{ne2} and \eqref{ne3} that
   \begin{align*}
\pi_2(\supp\mu) = \{\IM(z_n)\colon n \Ge 1\} \subset
\Big[0,\frac{b}{a}\Big].
   \end{align*}
Since, by \eqref{ne1}, $\IM(z_n) \searrow 0$, we see
that $\overline{\pi_2(\supp\mu)}= \{0\} \cup
\{\IM(z_n)\colon n \Ge 1\}$. This implies that
   \begin{align} \label{ne4}
\zcal_p \cap \big(\rbb \times
\overline{\pi_2(\supp\mu)}\,\big) = \{z_n\colon n\Ge
1\} \cup \{- \bar z_n\colon n\Ge 1\}.
   \end{align}
Let $\tilde \mu$ be any representing measure for
$\gammab$. In view of \eqref{ne4} and Proposition
\ref{fontan2}, the measure $\tilde \mu$ satisfies
\eqref{ne6}. It follows from Lemma \ref{abol} and
Proposition \ref{fontan} that $\tilde\mu \circ
\pi_2^{-1} = \mu \circ \pi_2^{-1}$. Since $\mu$ and
$\tilde \mu$ are supported in $\zcal_p$ and
   \begin{align*}
\zcal_p \cap \pi_2^{-1}(\{\IM(z_n)\})=\{z_n,- \bar
z_n\}, \quad n\Ge 1,
   \end{align*}
we conclude that \eqref{ne5} holds. This completes the
proof.
   \end{proof}
   \begin{rem}
We conclude this section by examining Borel
injectivity of $\psi_p$ after transformation by
polynomial automorphism. For simplicity of
presentation we treat $\cbb$ as $\rbb^2$. Let us
consider a polynomial automorphism\footnote{\;See
\cite{vdE} for fundamentals of the theory of
polynomial automorphisms.} $\varPhi\colon \rbb^2 \to
\rbb^2$ given by
   \begin{align*}
\varPhi(x,y) = (x, y + f(x)), \quad x,y \in \rbb,
   \end{align*}
where $f\in \rbb[x]$. Clearly, the inverse of
$\varPhi$ is given by
   \begin{align*}
\varPhi^{-1}(x,y) = (x, y - f(x)), \quad x,y \in \rbb.
   \end{align*}
Though polynomial automorphisms preserve many
properties of moment sequences (see e.g., \cite[Sec.\
21]{st-sz1} or \cite[Proposition 46]{c-s-s-ap}), they
fail to preserve injectivity of $\psi_p$. Indeed, if
$p(x,y)= y - 1$ for $x,y\in \rbb$, then $\psi_p$ is
injective (see Example \ref{wiele}\ding{192}). Note
that $\varPhi(\zcal_p) = \zcal_{p\circ \varPhi^{-1}}$
and $p\circ \varPhi^{-1}(x,y)= y-f(x)-1$ for $x,y\in
\rbb$. Let $f(x)=x^2$ for $x\in \rbb$. Then
$\varPhi(\zcal_p)$ is the parabola $y=x^2 + 1$, which
means that $\psi_{p\circ \varPhi^{-1}}$ is not
injective. \hfill$\triangleleft$
   \end{rem}
   \section{\label{Sec6}An open problem}
The following question, partially answered in Theorem
\ref{null} and Corollary \ref{nullab}, needs to be
solved in full generality.
   \begin{opqq}
Assume that $\gammab\colon \nfr \to \cbb$ is a complex
moment sequence such that $\PDE(\gammab) =
\{\Gammab\}$. Does it follow that $\gammab$ is
determinate?
   \end{opqq}
In view of this question it is legitimate to make sure
that none of the examples given in this paper solves
it in the negative.
   \begin{rem} \label{reM-2}
The sequence $\gammab$ from Example \ref{no-atom} is
indeterminate and, by Theorem \ref{0notatom}(v) and
Example \ref{wiele}\ding{192}, the set $\PDE(\gammab)$
is infinite.

We will show that the same conclusion holds for the
indeterminate complex moment sequence
$\gammab=\{\gamma_{m,n}\}_{(m,n)\in \nfr}$ coming from
the two-dimensional Hamburger moment sequence
$\sbold\otimes\tbold$ appearing in Theorem \ref{DC1}
if\,\footnote{\;A similar argument can be applied to
the case when $\supp\mu$ is compact.}
   \begin{align} \label{doinf}
\text{$d\okr \sup\supp\mu \in (0,\infty)$ and
$\supp\mu \subset [0,d]$,}
   \end{align}
where $\mu$ is as in Theorem \ref{DC1}. To be more
precise, the complex moment sequence $\gammab$ is
defined by (cf.\ \eqref{baba})
   \begin{align*}
\gamma_{m,n} = \int_{\rbb^2} (x+\I y)^m (x-\I y)^n
\D\varrho(x,y), \quad m,n \Ge 0,
   \end{align*}
where $\varrho$ is any representing measure for
$\sbold\otimes\tbold$. The definition of $\gammab$ is
independent of the choice of $\varrho$, and, after
identifying $\cbb$ with $\rbb^2$, representing
measures of $\sbold\otimes\tbold$ and $\gammab$
coincide (see \cite[Appendix A]{c-s-s-rh}). We will
indicate two representing measures $\varrho_1$ and
$\varrho_2$ for $\gammab$ (equivalently for
$\sbold\otimes\tbold$) such that
   \begin{align}\label{nopis}
\varrho_1\circ\psi^{-1} \neq \varrho_2\circ\psi^{-1}.
   \end{align}
For this, we first observe that if $\varDelta\subset
(0,\pi)$ is an open interval and
   \begin{align*}
E_\varDelta \okr \{\pm r\E^{\I t}\colon r>0, t\in
\varDelta\},
   \end{align*}
then
   \begin{align} \label{nopis2}
E_\varDelta=\psi^{-1}(\{\E^{2\I t}\colon
t\in\varDelta\}).
   \end{align}
Next, we notice that it is possible to find two
representing measures $\nu_1$ and $\nu_2$ for $\tbold$
for which there exist $b_1,b_2\in\rbb$ such that
   \begin{align} \label{bbbb}
\text{$b_2>b_1>0$, $b_1 \in \supp \nu_1$ and
$\nu_2((0,b_2))=0$.}
   \end{align}
Indeed, this is always true for any two distinct
N-extremal measures of $\tbold$ having atoms in
$(0,\infty)$ (up to rearrangement), because supports
of N-extremal measures of $\tbold$ form the partition
of $\rbb$ and each of them has no accumulation points
in $\rbb$ (see \cite[Theorem 2.13]{Sh-Tam}). Set
$\varrho_1 = \mu\otimes\nu_1$ and $\varrho_2 =
\mu\otimes\nu_2$. Then $\varrho_1$and $\varrho_2$ are
representing measures for $\gammab$ (see the proof of
Proposition \ref{DC1}). Moreover, by \eqref{doinf},
\eqref{bbbb} and $\mu(\nul)=0$, we have
   \begin{align} \label{varom}
\varrho_2(\varOmega)=0,
   \end{align}
where
   \begin{align*}
\varOmega = \Big((-\infty, 0]\times \rbb\Big)\cup
\Big([0,d]\times (0,b_2)\Big) \cup \Big(
(d,\infty)\times \rbb\Big),
   \end{align*}
and
   \begin{align} \label{dbjed}
(d,b_1)\in \supp \mu \times \supp \nu_1 =
\supp\varrho_1.
   \end{align}
Plainly, we can choose a set $E_\varDelta$, where
$\varDelta\subset (0,\pi)$ is an open interval, so
that $E_\varDelta \subset \varOmega$ and $(d,b_1)\in
E_\varDelta$. This combined with \eqref{varom},
\eqref{dbjed} and the fact that $E_\varDelta$ is an
open neighbourhood of $(d,b_1)$ implies that
$\varrho_2(E_\varDelta)= 0$ and
$\varrho_1(E_\varDelta)> 0$. Hence, by \eqref{nopis2},
we get \eqref{nopis}. Since, by Theorem \ref{DC1},
none of representing measures of $\gammab$ has an atom
at $0$, we infer from Lemma \ref{repr2}(ii) that
$(\varrho_1,0)$ and $(\varrho_2,0)$ are representing
pairs for some extensions in $\PDE(\gammab)$. It
follows from Lemma \ref{chmura2} and \eqref{nopis}
that they cannot be representing pairs for the same
$\Gammab\in\PDE(\gammab)$. This implies that the set
$\PDE(\gammab)$ is infinite (see the proof of
Proposition \ref{snu2}).
   \hfill$\triangleleft$
   \end{rem}

\vspace{1ex}

{\bf Acknowledgments.} The authors are grateful to the
referee for suggestions that helped to improve the
final version of the paper.

   \bibliographystyle{amsalpha}
   
   \end{document}